\documentclass[12pt]{article}

\usepackage[T1]{fontenc}
\usepackage{lmodern}
\usepackage{amsmath,amssymb,amsfonts,amsthm,mathtools}
\usepackage{bm}
\usepackage{tikz}
\usetikzlibrary{arrows.meta,calc,positioning}
\usepackage{enumitem}
\usepackage{hyperref}
\usepackage[noabbrev]{cleveref}
\usepackage{microtype}

\usepackage[utf8]{inputenc}

\usepackage[left=1.05in,right=1.05in,top=1in,bottom=1in]{geometry}
\usepackage{amsmath,amssymb,amsfonts,amsthm,mathtools}
\usepackage{graphicx}
\usepackage{float}
\usepackage{tikz}
\usetikzlibrary{arrows.meta,calc,positioning,decorations.markings}

\hypersetup{
  colorlinks=true,
  linkcolor=blue!50!black,
  citecolor=blue!60!black,
  urlcolor=blue!60!black
}

\numberwithin{equation}{section}

\newcommand{\R}{\mathbb{R}}
\newcommand{\C}{\mathbb{C}}

\numberwithin{equation}{section}
\newtheorem{theorem}{Theorem}[section]
\newtheorem{proposition}[theorem]{Proposition}

\newtheorem{corollary}[theorem]{Corollary}
\theoremstyle{definition}
\newtheorem{definition}[theorem]{Definition}
\newtheorem{remark}[theorem]{Remark}
\newtheorem{example}[theorem]{Example}

\newcommand{\CP}{\mathbb{CP}}

\newcommand{\dd}{\,d}
\newcommand{\abs}[1]{\left\lvert #1 \right\rvert}

\newcommand{\ip}[2]{\left\langle #1,#2 \right\rangle}

\title{Explicit Minimal Surface Models in $\mathbb{R}^5$ via Holomorphic Null Curves}
\author{Magdalena Toda, Erhan Güler}
\date{}

\begin{document}
\maketitle

\begin{abstract}
We study explicit conformal minimal immersions into $\mathbb{R}^5$ obtained from holomorphic null curves in $\mathbb{C}^5$. Although the general correspondence between conformal minimal immersions in $\mathbb{R}^n$ and holomorphic null data in $\mathbb{C}^n$ is classical, our aim here is different. We isolate the five-dimensional case and develop a concrete, self-contained account that emphasizes explicit formulas, integral-free constructions, and coordinate expressions suitable for computation and visualization.

Starting from a Weierstrass-type representation in $\mathbb{R}^5$, we derive a family of conformal minimal immersions depending on a single holomorphic seed function and two real parameters. The resulting formulas allow the immersion and the induced metric to be written in closed form. We then examine polynomial seeds in detail, derive their polar and Cartesian expansions, and discuss the geometric information carried by natural coordinate projections.

We reinterpret the construction in the language of moving frames, the generalized Gauss map, and a local DPW-type scheme. This provides a conceptual bridge between explicit holomorphic formulas and the Cartan-integrable-systems viewpoint. The discussion is local and formula-driven; global questions such as periods, completeness, and embeddedness lie beyond the present scope.

We also briefly clarify why the complex-analytic structure underlying the representation is essential, and why it cannot be replaced by a naive quaternionic formalism, due to the loss of commutativity, holomorphic structure, and compatibility with the null-curve framework.

\end{abstract}


\tableofcontents

\section{Introduction}

The classical Weierstrass-Enneper representation is one of the most beautiful and useful bridges between complex analysis and differential geometry. In its familiar three-dimensional form, a conformal minimal immersion can be described locally as the real part of a holomorphic null curve in $\mathbb{C}^3$, as developed in the classical treatments of Osserman \cite{Osserman2}, Nitsche \cite{Nitsche}, and earlier foundational work of Osserman \cite{Osserman1}. This representation not only gives a conceptual characterization of minimal surfaces, but also produces explicit parametrizations and a large supply of examples; see also the comprehensive exposition by Dierkes et al. \cite{Dierkes}.

The same general principle extends to higher-dimensional Euclidean spaces. If one replaces $\mathbb{C}^3$ by $\mathbb{C}^n$ and imposes the corresponding quadratic nullity condition, then the real part of a holomorphic primitive yields a conformal minimal immersion into $\mathbb{R}^n$, as shown by Chern and Osserman \cite{ChernOsserman} and Beckenbach \cite{Beckenbach}. Alternative formulations and representation results for prescribed mean curvature surfaces, such as those of Kenmotsu \cite{Kenmotsu}, further illustrate the flexibility of this complex-analytic framework. In hyperbolic space, analogous representation formulas were developed by Bryant \cite{Bryant} and later generalized by Aiyama and Akutagawa \cite{AiyamaAkutagawa}.

From a theoretical point of view this extension is classical. From a practical point of view, however, explicit calculations become increasingly difficult as the ambient dimension grows. The formulas become longer, the geometry less visible, and the connection between abstract theory and concrete examples is harder to maintain.

The purpose of this paper is to focus entirely on the case of $\mathbb{R}^5$. This dimension is especially attractive for a formula-based treatment. It is the first case beyond $\mathbb{R}^4$ in which the codimension is large enough to display several independent normal directions, yet still low enough that one can derive closed formulas, compute induced metrics directly, and write down examples in a form suitable for plotting and geometric discussion.

Our emphasis is therefore not on greater generality, but on clarity and explicitness. We begin with a Weierstrass-type formula adapted to $\mathbb{R}^5$ and then specialize it in such a way that the primitive can be written without any remaining integration. This approach is related to recent integral-free constructions and explicit representation formulas obtained by G\"uler and Toda \cite{TodaGulerRn} and by Toda and G\"uler \cite{TodaGulerR4}. This leads to a family of minimal surfaces determined by a single holomorphic seed function and two real parameters. Once this family is established, the induced metric can be computed in a clean closed form, and the generalized Gauss map becomes completely explicit in the spirit of Hoffman and Osserman \cite{HoffmanOsserman}.

The polynomial case is of particular interest. When the seed is chosen to be a monomial, the resulting coordinates of the immersion involve several different powers of the complex variable, and this produces nontrivial geometry already at a very elementary level. Since a surface in $\mathbb{R}^5$ cannot be visualized directly, we also discuss how coordinate projections can be used systematically to reveal different parts of the geometry.

In the final section, we step back from the explicit formulas and explain how they fit into the language of moving frames and harmonic maps into symmetric spaces, following the viewpoints of Pressley and Segal \cite{PressleySegal} and Dajczer and Tojeiro \cite{DajczerTojeiro}. This leads naturally to a local DPW-type perspective introduced by Dorfmeister, Pedit, and Wu \cite{DPW}. We do not attempt to develop a full global integrable-systems theory here, but it is useful to show that the concrete formulas treated in the first part of the paper sit inside a broader geometric framework.

Throughout the paper the discussion remains local. We do not address global period problems, completeness, embeddedness, or classification issues. Our goal is more modest and more concrete: to provide a readable and computationally useful account of explicit minimal surfaces in $\mathbb{R}^5$.

Finally, we note that the present construction is intrinsically complex-analytic. It may be tempting to seek an analogous formulation using quaternionic variables, especially in view of the higher-dimensional setting. However, such an approach is not compatible with the structure of the representation developed here. The key ingredients of the theory -- namely the holomorphic null curve, the quadratic complex-bilinear nullity condition, and the identity $X_z=\Phi$, $X_{\bar z}=0$ -- all rely essentially on the commutative and holomorphic nature of complex multiplication. These features break down in the quaternionic setting, where noncommutativity and the absence of a comparable primitive theory prevent a direct reformulation. For clarity, we return to this point in a dedicated discussion later in the paper.

\section{A Weierstrass-Type Formula in $\mathbb{R}^5$}

Let $D \subset \mathbb{C}$ be a simply connected domain with complex coordinate
\[
z = u + iv.
\]
Let
\[
f, g_1, g_2, g_3 : D \to \mathbb{C}
\]
be holomorphic functions. We define a holomorphic $\mathbb{C}^5$-valued map
\[
\Phi(z) = (\phi_1(z),\phi_2(z),\phi_3(z),\phi_4(z),\phi_5(z))
\]
by
\begin{align}
\phi_1 &= \frac{1}{2} f \bigl(1 - g_1^2 - g_2^2 - g_3^2\bigr), \label{phi1}\\
\phi_2 &= \frac{i}{2} f \bigl(1 + g_1^2 + g_2^2 + g_3^2\bigr), \label{phi2}\\
\phi_3 &= f g_1, \qquad \phi_4 = f g_2, \qquad \phi_5 = f g_3. \label{phi345}
\end{align}
The first point to verify is that this data is automatically null.

\begin{theorem}\label{thm:null}
The components of $\Phi$ satisfy
\[
\phi_1^2+\phi_2^2+\phi_3^2+\phi_4^2+\phi_5^2=0
\]
identically on $D$.
\end{theorem}

\begin{proof}
Set
\[
\sigma = g_1^2 + g_2^2 + g_3^2.
\]
Then \eqref{phi1}--\eqref{phi345} become
\[
\phi_1 = \frac{1}{2}f(1-\sigma), \qquad
\phi_2 = \frac{i}{2}f(1+\sigma), \qquad
\phi_3 = fg_1, \qquad
\phi_4 = fg_2, \qquad
\phi_5 = fg_3.
\]
Hence
\[
\phi_1^2 = \frac{1}{4}f^2(1-\sigma)^2,
\qquad
\phi_2^2 = -\frac{1}{4}f^2(1+\sigma)^2,
\]
while
\[
\phi_3^2+\phi_4^2+\phi_5^2
= f^2(g_1^2+g_2^2+g_3^2)
= f^2\sigma.
\]
Therefore
\begin{align*}
\phi_1^2+\phi_2^2+\phi_3^2+\phi_4^2+\phi_5^2
&=
\frac{1}{4}f^2\bigl((1-\sigma)^2-(1+\sigma)^2\bigr)+f^2\sigma\\
&=
\frac{1}{4}f^2(-4\sigma)+f^2\sigma\\
&=0.
\end{align*}
This proves the identity.
\end{proof}

Because the components of $\Phi$ are holomorphic, the $\mathbb{C}^5$-valued differential $\Phi(z)\,\mathrm{d}z$ has a primitive on the simply connected domain $D$. We therefore define
\begin{equation}\label{Xdef}
X(z)=2\,\Re \int_{z_0}^{z}\Phi(\zeta)\,\mathrm{d}\zeta,
\end{equation}
where $z_0\in D$ is fixed. The factor $2$ is included so that the metric formulas that follow take a particularly clean form.

\begin{definition}
The map $X:D\to \mathbb{R}^5$ defined by \eqref{Xdef} is called the conformal minimal immersion associated with the holomorphic null data $(f,g_1,g_2,g_3)$, provided its differential does not vanish.
\end{definition}

\begin{theorem}\label{thm:minimal}
The map $X$ defined by \eqref{Xdef} is harmonic and conformal wherever $\Phi\neq 0$. In particular, it is a conformal minimal immersion on its regular set.
\end{theorem}

\begin{proof}
Let
\[
F(z)=\int_{z_0}^{z}\Phi(\zeta)\,\mathrm{d}\zeta.
\]
Then $F$ is holomorphic and $X=2\,\Re F$. Each coordinate of $X$ is therefore the real part of a holomorphic function, hence harmonic. Thus $X$ is harmonic.

To prove conformality, we differentiate in complex notation. Since $F_z=\Phi$ and $F_{\bar z}=0$, we obtain
\[
X_z = \Phi,
\qquad
X_{\bar z}=\overline{\Phi}.
\]
The complex bilinear inner product on $\mathbb{C}^5$, then gives
\[
\ip{X_z}{X_z}
=
\sum_{k=1}^5 \phi_k^2
=0
\]
by Theorem \ref{thm:null}. This is exactly the conformality condition. Since a conformal harmonic immersion is minimal, the proof is complete.
\end{proof}

\begin{remark}
The mechanism is the same as in the classical three-dimensional case. Holomorphicity produces harmonicity, and the null condition produces conformality. The essential new feature is that the null vector now has five components, and the last three carry the genuinely higher-codimensional information.
\end{remark}

\section{An Integral-Free Construction from a Single Seed Function}

The representation \eqref{Xdef} is conceptually natural, but it still involves integration. For explicit examples it is often useful to remove the integral entirely and write the immersion directly in terms of a holomorphic seed function and its derivatives.

Let $\Psi$ be holomorphic on $D$. Let $\lambda_1,\lambda_2\in\R$ be fixed real parameters, and define
\[
\Lambda = 1+\lambda_1^2+\lambda_2^2.
\]
We choose the Weierstrass data
\begin{equation}\label{seeddata}
f(z)=\Psi'''(z), \qquad
g_1(z)=z, \qquad
g_2(z)=\lambda_1 z, \qquad
g_3(z)=\lambda_2 z.
\end{equation}
Substituting \eqref{seeddata} into \eqref{phi1}--\eqref{phi345}, we obtain
\begin{align}
\phi_1 &= \frac{1}{2}\Psi'''(z)\bigl(1-\Lambda z^2\bigr), \label{phiseed1}\\
\phi_2 &= \frac{i}{2}\Psi'''(z)\bigl(1+\Lambda z^2\bigr), \label{phiseed2}\\
\phi_3 &= z\Psi'''(z), \qquad
\phi_4 = \lambda_1 z\Psi'''(z), \qquad
\phi_5 = \lambda_2 z\Psi'''(z). \label{phiseed345}
\end{align}

We now seek a primitive of this vector field that can be written without any remaining integral signs. Define $F=(F_1,\dots,F_5):D\to\C^5$ by
\begin{align}
F_1(z) &= \frac{1}{2}(1-\Lambda z^2)\Psi''(z)+\Lambda z\Psi'(z)-\Lambda\Psi(z), \label{F1}\\
F_2(z) &= \frac{i}{2}(1+\Lambda z^2)\Psi''(z)-i\Lambda z\Psi'(z)+i\Lambda\Psi(z), \label{F2}\\
F_3(z) &= z\Psi''(z)-\Psi'(z), \label{F3}\\
F_4(z) &= \lambda_1\bigl(z\Psi''(z)-\Psi'(z)\bigr), \label{F4}\\
F_5(z) &= \lambda_2\bigl(z\Psi''(z)-\Psi'(z)\bigr). \label{F5}
\end{align}

\begin{proposition}\label{prop:Fprime}
The function $F$ satisfies
\[
F'(z)=\Phi(z),
\]
where $\Phi$ is defined by \eqref{phiseed1}--\eqref{phiseed345}.
\end{proposition}

\begin{proof}
We differentiate each component directly. For $F_3$ one has
\[
F_3'(z)=\Psi''(z)+z\Psi'''(z)-\Psi''(z)=z\Psi'''(z)=\phi_3(z).
\]
The formulas for $F_4$ and $F_5$ follow immediately from the same computation.

For $F_1$, we write
\[
F_1(z)=\frac{1}{2}(1-\Lambda z^2)\Psi''(z)+\Lambda z\Psi'(z)-\Lambda\Psi(z).
\]
Differentiating term by term yields
\[
\frac{\dd}{\dd z}\left[\frac{1}{2}(1-\Lambda z^2)\Psi''(z)\right]
=
-\Lambda z\Psi''(z)+\frac{1}{2}(1-\Lambda z^2)\Psi'''(z),
\]
and
\[
\frac{\dd}{\dd z}\bigl(\Lambda z\Psi'(z)-\Lambda\Psi(z)\bigr)
=
\Lambda\Psi'(z)+\Lambda z\Psi''(z)-\Lambda\Psi'(z)
=
\Lambda z\Psi''(z).
\]
The terms involving $\Psi''$ cancel, leaving
\[
F_1'(z)=\frac{1}{2}(1-\Lambda z^2)\Psi'''(z)=\phi_1(z).
\]
The computation for $F_2$ is identical except for the factor $i$, and gives
\[
F_2'(z)=\frac{i}{2}(1+\Lambda z^2)\Psi'''(z)=\phi_2(z).
\]
Thus $F'=\Phi$.
\end{proof}

Combining Proposition \ref{prop:Fprime} with \eqref{Xdef}, we obtain the following integral-free formulation.

\begin{corollary}\label{cor:integralfree}
With the notation above, the immersion is given by
\[
X(z)=2\,\Re F(z).
\]
In particular, the immersion depends only on $\Psi$, $\Psi'$, and $\Psi''$.
\end{corollary}

The practical value of this formulation is immediate. Instead of first constructing the null curve and then integrating it, one can work directly with the seed function. This is especially useful for polynomial or rational examples, where the coordinates of $X$ can be written down explicitly in closed form.

There is also a simple identity that recovers the seed function from the components of $F$.

\begin{proposition}\label{prop:reconstruction}
With the notation of \eqref{F1}--\eqref{F5}, one has
\begin{equation}\label{reconstruction}
\Psi(z)
=
\frac{\Lambda z^2-1}{2\Lambda}F_1(z)
-
\frac{i(\Lambda z^2+1)}{2\Lambda}F_2(z)
-
\frac{z}{\Lambda}\bigl(F_3(z)+\lambda_1F_4(z)+\lambda_2F_5(z)\bigr).
\end{equation}
\end{proposition}

\begin{proof}
From \eqref{F3}--\eqref{F5} we obtain
\[
F_3+\lambda_1F_4+\lambda_2F_5
=
(1+\lambda_1^2+\lambda_2^2)\bigl(z\Psi''-\Psi'\bigr)
=
\Lambda(z\Psi''-\Psi').
\]
Next, using \eqref{F1} and \eqref{F2}, we compute
\begin{align*}
\frac{\Lambda z^2-1}{2\Lambda}F_1
&=
\frac{\Lambda z^2-1}{4\Lambda}(1-\Lambda z^2)\Psi''
+\frac{\Lambda z^2-1}{2}z\Psi'
-\frac{\Lambda z^2-1}{2}\Psi,\\
-\frac{i(\Lambda z^2+1)}{2\Lambda}F_2
&=
\frac{\Lambda z^2+1}{4\Lambda}(1+\Lambda z^2)\Psi''
-\frac{\Lambda z^2+1}{2}z\Psi'
+\frac{\Lambda z^2+1}{2}\Psi.
\end{align*}
Adding these two expressions, the coefficient of $\Psi''$ becomes
\[
\frac{(\Lambda z^2-1)(1-\Lambda z^2)+(\Lambda z^2+1)(1+\Lambda z^2)}{4\Lambda}
=
\frac{2\Lambda^2 z^4+2}{4\Lambda}
=
\frac{\Lambda z^4+1/\Lambda}{2},
\]
while the coefficient of $\Psi'$ becomes
\[
\frac{\Lambda z^2-1-(\Lambda z^2+1)}{2}z=-z,
\]
and the coefficient of $\Psi$ becomes
\[
-\frac{\Lambda z^2-1}{2}+\frac{\Lambda z^2+1}{2}=1.
\]
Subtracting
\[
\frac{z}{\Lambda}\bigl(F_3+\lambda_1F_4+\lambda_2F_5\bigr)
=
z^2\Psi''-z\Psi'
\]
cancels the remaining $\Psi''$ and $\Psi'$ contributions, leaving exactly $\Psi$. This proves \eqref{reconstruction}.
\end{proof}

\begin{remark}
The reconstruction identity is not needed for the basic construction, but it shows that the seed function is encoded algebraically in the primitive $F$. In that sense the integral-free parametrization does not lose information.
\end{remark}

\section{The Induced Metric and the Generalized Gauss Map}

We now compute the induced metric of the immersion associated with \eqref{seeddata}. Since $X_z=\Phi$, the first fundamental form is
\begin{equation}\label{metricgeneral}
ds^2 = 2\,\ip{X_z}{X_{\bar z}}\,|dz|^2
      = 2\sum_{k=1}^5 \abs{\phi_k}^2 |dz|^2.
\end{equation}
The advantage of the normalization chosen in \eqref{Xdef} is that the factor in front of the sum is exactly $2$.

\begin{theorem}\label{thm:metric}
For the seed construction \eqref{seeddata}, the induced metric is
\begin{equation}\label{metricformula}
ds^2 = \abs{\Psi'''(z)}^2 \bigl(1+\Lambda \abs{z}^2\bigr)^2 |dz|^2.
\end{equation}
\end{theorem}

\begin{proof}
Using \eqref{phiseed1}--\eqref{phiseed345}, we obtain
\begin{align*}
\sum_{k=1}^5 \abs{\phi_k}^2
&=
\frac{1}{4}\abs{\Psi'''}^2 \abs{1-\Lambda z^2}^2
+
\frac{1}{4}\abs{\Psi'''}^2 \abs{1+\Lambda z^2}^2
+
\abs{\Psi'''}^2 \abs{z}^2 \\
&\qquad
+
\lambda_1^2\abs{\Psi'''}^2 \abs{z}^2
+
\lambda_2^2\abs{\Psi'''}^2 \abs{z}^2.
\end{align*}
Since $\Lambda$ is real, we use the identity
\[
\abs{1-\Lambda z^2}^2+\abs{1+\Lambda z^2}^2
=
2\bigl(1+\Lambda^2\abs{z}^4\bigr).
\]
Hence
\[
\sum_{k=1}^5 \abs{\phi_k}^2
=
\frac{1}{2}\abs{\Psi'''}^2\bigl(1+\Lambda^2\abs{z}^4\bigr)
+
(1+\lambda_1^2+\lambda_2^2)\abs{\Psi'''}^2\abs{z}^2.
\]
Because $\Lambda=1+\lambda_1^2+\lambda_2^2$, this becomes
\[
\sum_{k=1}^5 \abs{\phi_k}^2
=
\frac{1}{2}\abs{\Psi'''}^2\bigl(1+2\Lambda\abs{z}^2+\Lambda^2\abs{z}^4\bigr)
=
\frac{1}{2}\abs{\Psi'''}^2 \bigl(1+\Lambda\abs{z}^2\bigr)^2.
\]
Substituting into \eqref{metricgeneral} yields
\[
ds^2
=
2\cdot \frac{1}{2}\abs{\Psi'''}^2 \bigl(1+\Lambda\abs{z}^2\bigr)^2 |dz|^2,
\]
which is exactly \eqref{metricformula}.
\end{proof}

\begin{remark}
The metric factorizes into two geometrically meaningful pieces. The term $\abs{\Psi'''(z)}^2$ is determined entirely by the seed function. The term $\bigl(1+\Lambda\abs{z}^2\bigr)^2$ records the effect of the three functions $g_1,g_2,g_3$, and hence the way the immersion is distributed among the higher-codimensional directions.
\end{remark}

The null curve also determines a generalized Gauss map. Since $\Phi(z)\neq 0$ on the regular set and satisfies the quadratic relation of Theorem \ref{thm:null}, the projective class $[\Phi(z)]$ lies in the smooth complex quadric
\[
Q_3
=
\left\{
[w_1:w_2:w_3:w_4:w_5]\in\CP^4
\; ; \;
w_1^2+w_2^2+w_3^2+w_4^2+w_5^2=0
\right\}.
\]
We therefore define
\begin{equation}\label{Gdef}
\mathcal{G}(z)
=
[\phi_1(z):\phi_2(z):\phi_3(z):\phi_4(z):\phi_5(z)]
\in Q_3.
\end{equation}

\begin{proposition}\label{prop:gaussmap}
The map $\mathcal{G}:D\to Q_3$ is holomorphic on the regular set of the immersion.
\end{proposition}

\begin{proof}
Each $\phi_k$ is holomorphic, and the null relation shows that the image lies in the quadric. Since projectivization preserves holomorphicity away from common zeros, the map $\mathcal{G}$ is holomorphic on the regular set.
\end{proof}

In the explicit family \eqref{seeddata}, the common factor $\Psi'''(z)$ cancels after projectivization, so the generalized Gauss map depends only on $g_1,g_2,g_3$. More precisely, \eqref{phiseed1}--\eqref{phiseed345} give
\begin{equation}\label{GaussExplicit}
\mathcal{G}(z)
=
\left[
1-\Lambda z^2 :
i(1+\Lambda z^2) :
2z :
2\lambda_1 z :
2\lambda_2 z
\right].
\end{equation}
Thus in the present family the generalized Gauss map is especially simple: it is determined entirely by the complex coordinate and the two real parameters.

\begin{remark}
The fact that the factor $\Psi'''$ disappears from \eqref{GaussExplicit} is important conceptually. The seed function controls the conformal factor of the metric and hence the scaling of the immersion, while the generalized Gauss map records only the projective direction of the null curve.
\end{remark}

\section{Polynomial Seeds and Explicit Coordinate Formulas}

We now study a natural and important class of examples. Let
\[
\Psi(z)=z^m,
\qquad m\geq 3.
\]
Then
\[
\Psi'(z)=m z^{m-1},
\qquad
\Psi''(z)=m(m-1)z^{m-2},
\qquad
\Psi'''(z)=m(m-1)(m-2)z^{m-3}.
\]
For convenience we set
\[
C_m = m(m-1)(m-2).
\]

Substituting these formulas into \eqref{F1}--\eqref{F5}, we obtain explicit polynomial expressions for the primitive $F$ and hence for the immersion $X=2\,\Re F$.

\begin{theorem}\label{thm:poly}
Let $\Psi(z)=z^m$ with $m\geq 3$. Then the primitive $F$ defined by \eqref{F1}--\eqref{F5} has components
\begin{align}
F_1(z)
&=
\frac{m(m-1)}{2}z^{m-2}
-
\frac{\Lambda}{2}(m-1)(m-2)z^m, \label{polyF1}\\
F_2(z)
&=
\frac{i\,m(m-1)}{2}z^{m-2}
+
\frac{i\Lambda}{2}(m-1)(m-2)z^m, \label{polyF2}\\
F_3(z)
&=
m(m-2)z^{m-1}, \label{polyF3}\\
F_4(z)
&=
\lambda_1 m(m-2)z^{m-1}, \label{polyF4}\\
F_5(z)
&=
\lambda_2 m(m-2)z^{m-1}. \label{polyF5}
\end{align}
Consequently, the immersion $X=2\,\Re F$ has induced metric
\begin{equation}\label{polymetric}
ds^2 = C_m^2 \abs{z}^{2m-6}\bigl(1+\Lambda\abs{z}^2\bigr)^2 |dz|^2.
\end{equation}
\end{theorem}

\begin{proof}
We compute the components one by one. For $F_3$,
\[
F_3(z)=z\Psi''(z)-\Psi'(z)
=
z\cdot m(m-1)z^{m-2} - m z^{m-1}
=
m(m-2)z^{m-1},
\]
which proves \eqref{polyF3}. The formulas \eqref{polyF4} and \eqref{polyF5} follow immediately by multiplying by $\lambda_1$ and $\lambda_2$.

For $F_1$, substituting $\Psi''(z)=m(m-1)z^{m-2}$ and $\Psi'(z)=m z^{m-1}$ into \eqref{F1} gives
\[
F_1(z)
=
\frac{1}{2}(1-\Lambda z^2)m(m-1)z^{m-2}
+\Lambda z\cdot m z^{m-1}
-\Lambda z^m.
\]
Expanding,
\[
F_1(z)
=
\frac{m(m-1)}{2}z^{m-2}
-\frac{\Lambda m(m-1)}{2}z^m
+\Lambda m z^m
-\Lambda z^m.
\]
The coefficient of $z^m$ simplifies to
\[
-\frac{\Lambda m(m-1)}{2}+\Lambda m-\Lambda
=
-\frac{\Lambda}{2}(m-1)(m-2),
\]
which proves \eqref{polyF1}.

The formula for $F_2$ is obtained similarly from \eqref{F2}:
\[
F_2(z)
=
\frac{i}{2}(1+\Lambda z^2)m(m-1)z^{m-2}
-i\Lambda z\cdot m z^{m-1}
+i\Lambda z^m,
\]
and simplifying the coefficient of $z^m$ gives \eqref{polyF2}.

Finally, \eqref{polymetric} follows from \eqref{metricformula} and the fact that
\[
\abs{\Psi'''(z)}^2 = C_m^2 \abs{z}^{2m-6}.
\]
\end{proof}

\begin{remark}
The structure of \eqref{polyF1}--\eqref{polyF5} is already geometrically suggestive. The first two coordinates involve two different degrees, namely $m-2$ and $m$, while the last three involve only degree $m-1$. This mixture of degrees is one of the reasons the projected images can be more interesting than one might expect from such a simple seed.
\end{remark}

To make the geometry more visible, it is helpful to write the immersion in polar coordinates. Let
\[
z = r e^{i\theta},
\qquad r\geq 0.
\]
Then
\[
z^{m-2}=r^{m-2}e^{i(m-2)\theta},
\qquad
z^{m-1}=r^{m-1}e^{i(m-1)\theta},
\qquad
z^m=r^m e^{im\theta}.
\]
Since $X=2\,\Re F$, Theorem \ref{thm:poly} gives the following formulas.

\begin{proposition}\label{prop:polar}
For the polynomial seed $\Psi(z)=z^m$, the immersion $X=(X_1,\dots,X_5)$ is given in polar coordinates by
\begin{align}
X_1(r,\theta)
&=
m(m-1)r^{m-2}\cos((m-2)\theta)
-
\Lambda (m-1)(m-2)r^m\cos(m\theta), \label{X1polar}\\
X_2(r,\theta)
&=
- m(m-1)r^{m-2}\sin((m-2)\theta)
-
\Lambda (m-1)(m-2)r^m\sin(m\theta), \label{X2polar}\\
X_3(r,\theta)
&=
2m(m-2)r^{m-1}\cos((m-1)\theta), \label{X3polar}\\
X_4(r,\theta)
&=
2\lambda_1 m(m-2)r^{m-1}\cos((m-1)\theta), \label{X4polar}\\
X_5(r,\theta)
&=
2\lambda_2 m(m-2)r^{m-1}\cos((m-1)\theta). \label{X5polar}
\end{align}
\end{proposition}

\begin{proof}
Each formula follows by taking twice the real part of the corresponding component in \eqref{polyF1}--\eqref{polyF5}. For instance,
\[
X_1=2\,\Re F_1
=
2\left[
\frac{m(m-1)}{2}r^{m-2}\cos((m-2)\theta)
-
\frac{\Lambda}{2}(m-1)(m-2)r^m\cos(m\theta)
\right],
\]
which simplifies to \eqref{X1polar}. The remaining cases are identical.
\end{proof}

\begin{example}[The quartic seed]\label{ex:quartic}
Take $\Psi(z)=z^4$. Then $m=4$, $C_4=24$, and Theorem \ref{thm:poly} gives
\[
F_1(z)=6z^2-\Lambda z^4,
\qquad
F_2(z)=6iz^2+i\Lambda z^4,
\qquad
F_3(z)=8z^3,
\]
together with
\[
F_4(z)=8\lambda_1 z^3,
\qquad
F_5(z)=8\lambda_2 z^3.
\]
Hence
\begin{align*}
X_1(r,\theta) &= 12r^2\cos(2\theta)-2\Lambda r^4\cos(4\theta),\\
X_2(r,\theta) &= -12r^2\sin(2\theta)-2\Lambda r^4\sin(4\theta),\\
X_3(r,\theta) &= 16r^3\cos(3\theta),\\
X_4(r,\theta) &= 16\lambda_1 r^3\cos(3\theta),\\
X_5(r,\theta) &= 16\lambda_2 r^3\cos(3\theta).
\end{align*}
The induced metric is
\[
ds^2 = 24^2 \abs{z}^2 (1+\Lambda\abs{z}^2)^2 |dz|^2.
\]
\end{example}

It is also useful to record Cartesian expansions for the quartic seed. Writing $z=u+iv$, one has
\[
z^2=(u^2-v^2)+2iuv,
\]
\[
z^3=(u^3-3uv^2)+i(3u^2v-v^3),
\]
and
\[
z^4=(u^4-6u^2v^2+v^4)+i(4u^3v-4uv^3).
\]
Substituting these into Example \ref{ex:quartic}, we obtain
\begin{align}
X_1(u,v)
&=
12(u^2-v^2)-2\Lambda(u^4-6u^2v^2+v^4), \label{X1cart}\\
X_2(u,v)
&=
-24uv-8\Lambda(u^3v-uv^3), \label{X2cart}\\
X_3(u,v)
&=
16(u^3-3uv^2), \label{X3cart}\\
X_4(u,v)
&=
16\lambda_1(u^3-3uv^2), \label{X4cart}\\
X_5(u,v)
&=
16\lambda_2(u^3-3uv^2). \label{X5cart}
\end{align}

\begin{remark}
In this family the last three coordinates are proportional. This reflects the special choice $g_2=\lambda_1 z$ and $g_3=\lambda_2 z$. More general choices of $g_1,g_2,g_3$ would produce greater variation among the final three coordinates, although at the cost of more complicated formulas.
\end{remark}

As shown in Example~\ref{ex:quartic}, the orthogonal projections onto the $X_1X_2X_3$-space of the polar (left) and Cartesian (right) surfaces in $\mathbb{R}^5$ are presented in Figure~\ref{fig:proj}.

\begin{figure}
    \centering
    \includegraphics[width=1\linewidth]{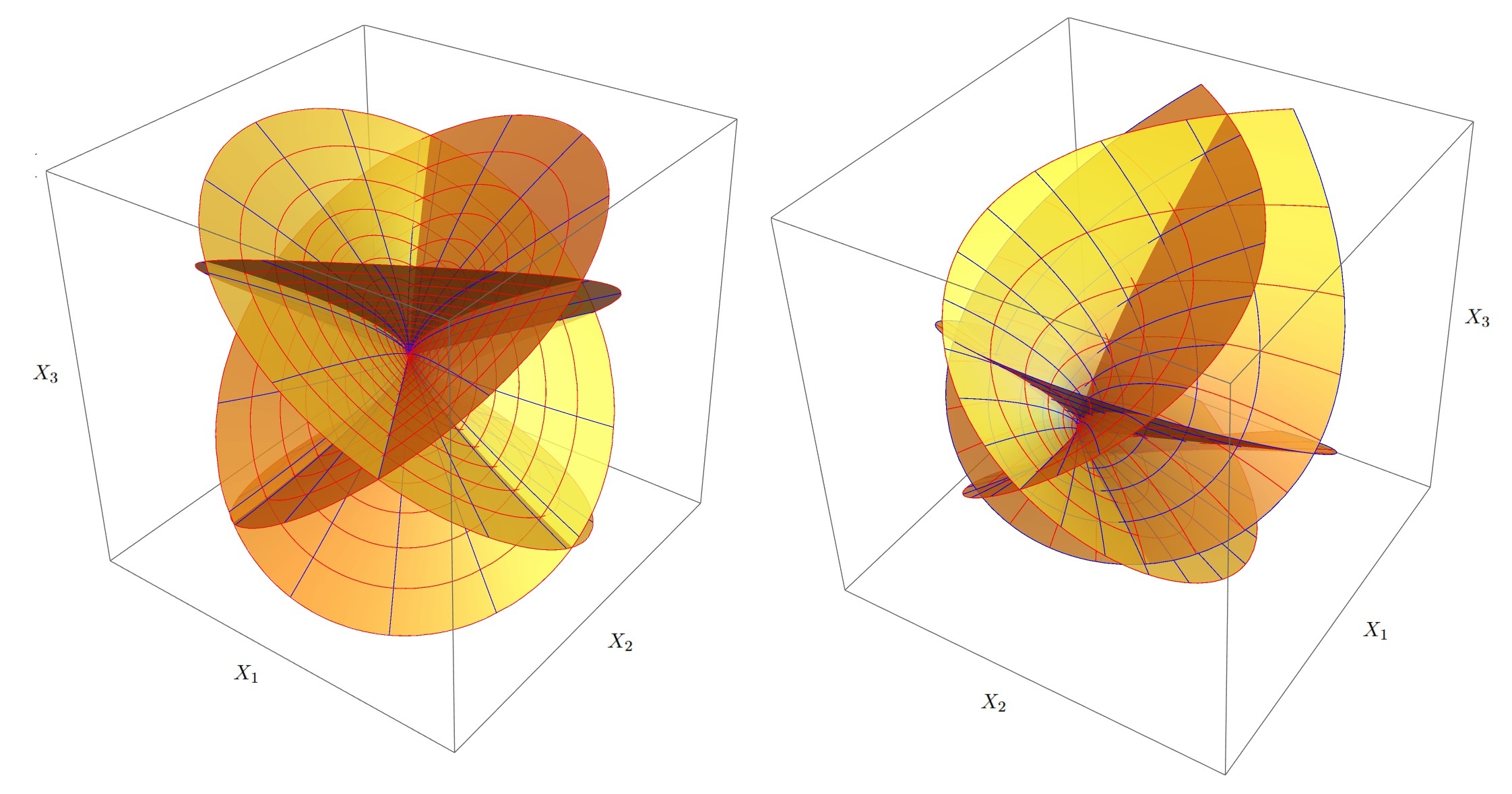}
    \caption{Orthogonal projections onto the $X_1X_2X_3$-space of the polar (left) and Cartesian (right) surfaces in $\mathbb{R}^5$, corresponding to Example~\ref{ex:quartic}, where $\lambda_1=3$ and $\lambda_2=5$.}
    \label{fig:proj}
\end{figure}

\section{Projection Geometry and Visualization}

Since the image of the immersion lies in $\mathbb{R}^5$, it cannot be visualized directly. One must therefore study coordinate projections. Although a projection does not preserve all geometric information, it often reveals patterns that are otherwise hidden.

Let
\[
X=(X_1,X_2,X_3,X_4,X_5):D\to \mathbb{R}^5
\]
be one of the minimal immersions constructed above. For each pair of indices $1\leq i<j\leq 5$, the map
\[
(X_i,X_j):D\to\mathbb{R}^2
\]
defines a planar projection. Since there are $\binom{5}{2}=10$ such pairs, there are ten canonical planar projections. Similarly, for each triple $1\leq i<j<k\leq 5$, the map
\[
(X_i,X_j,X_k):D\to\R^3
\]
defines a three-dimensional projection surface, and again there are $\binom{5}{3}=10$ such triples.

The planar projection $(X_1,X_2)$ is particularly natural. These two coordinates come from the first complex pair in the Weierstrass data and therefore retain the most direct trace of the oscillatory behavior inherited from the powers of $z$. When the seed is polynomial, the projection $(X_1,X_2)$ mixes the frequencies $(m-2)\theta$ and $m\theta$, and this can produce multi-lobed or rotating patterns even before one looks at the higher-codimensional directions.

The coordinates $(X_3,X_4,X_5)$ play a different role. They show how the surface spreads among the three directions beyond the first complex pair. In our specific family they are proportional, so their image lies in a two-dimensional affine subspace of $\R^3$. Nevertheless, when these coordinates are combined with $X_1$ or $X_2$ in mixed projections such as $(X_1,X_3,X_4)$ or $(X_2,X_4,X_5)$, one obtains a more informative picture of the ambient five-dimensional geometry.

For practical plotting, the polar formulas \eqref{X1polar}--\eqref{X5polar} are often the most convenient. One fixes a bounded radial interval $0\leq r\leq R$ and samples the angles uniformly. The quartic seed of Example \ref{ex:quartic} is especially suitable because the frequencies $2\theta$, $3\theta$, and $4\theta$ are low enough to produce visually rich but still readable patterns.

It is important to emphasize, however, that these projections are only visualization devices. In general, projecting a minimal surface from $\mathbb{R}^5$ to a lower-dimensional Euclidean space does not preserve minimality. The mean curvature vector depends on the ambient second fundamental form, and this changes under projection.

The following simple observation explains the one important exception.

\begin{proposition}\label{prop:projectionminimal}
Let $X:D\to\mathbb{R}^n$ be a conformal minimal immersion whose image is contained in an affine subspace $A\subset\mathbb{R}^n$ of dimension three. Then, after identifying $A$ isometrically with $\R^3$, the image is a minimal surface in $\R^3$.
\end{proposition}

\begin{proof}
If the image of $X$ is contained in $A$, then $X$ may be viewed as an immersion into the Euclidean space $A$ itself. Since the inclusion of $A$ into $\mathbb{R}^n$ is totally geodesic and isometric, the second fundamental form of the immersion computed in $A$ agrees with the tangential part of the second fundamental form computed in $\mathbb{R}^n$. Hence vanishing of the mean curvature vector in $\mathbb{R}^n$ implies vanishing of the mean curvature vector in $A$. Identifying $A$ with $\R^3$ gives the conclusion.
\end{proof}

\begin{remark}
Outside such lower-dimensional degeneracies, a projected image should not be interpreted as a minimal surface in its own right. What it preserves is not minimality, but certain visible traces of the higher-dimensional geometry.
\end{remark}

\medskip
In order to better understand the algebraic complexity underlying these projections, it is natural to examine the implicit equations satisfied by the coordinate functions. Indeed, eliminating the parameters from the Cartesian representation leads to a high-degree polynomial relation among the variables $(X_1,X_2,X_3)$.

This mechanism is already illustrated in Example~\ref{ex:quartic}, where eliminating the parameters $u$ and $v$ from the first three coordinate functions yields an explicit algebraic relation among $(X_1,X_2,X_3)$. Motivated by this observation, we are led to consider equations of the form
\[
F(X_1,X_2,X_3)=0,
\]
where $F$ is a multivariate polynomial consisting of monomials of varying degrees.

In general, such polynomials are of high degree and involve a large number of terms. For instance, the leading part of $F$ can be written as
\begin{equation}\label{eq:F}
\begin{aligned}
F &= \Lambda^{12}X_3^{16} 
+ 4096\,\Lambda^{9}X_1^{3}X_3^{12}
- 12288\,\Lambda^{9}X_1X_2^{2}X_3^{12} \\
&\quad 
- 165888\,\Lambda^{9}X_3^{14} 
- 4202496\,\Lambda^{8}X_1^{2}X_3^{12}\\
&\quad + \text{(remaining 73 lower-degree terms)}.
\end{aligned}
\end{equation}

\medskip
Figure~\ref{fig:projectionX123} illustrates the graph of \eqref{eq:F} in the three-dimensional $X_1X_2X_3$-space.

\medskip

The coexistence of high-degree terms and purely spatial monomials reflects the interaction between scaling effects and angular oscillations. From a geometric perspective, the highest-degree terms govern the asymptotic behavior of the projection, whereas the lower-degree terms encode finer structures such as lobes, symmetries, and self-intersections in the projected images.

Although such implicit equations are typically too large to analyze term-by-term, even partial truncations reveal an important feature: the geometry of the projected image is governed by a delicate balance between different frequency components. This explains why relatively simple polynomial seeds can generate highly intricate patterns when visualized through coordinate projections.

\section{Some Elementary Geometric Remarks}

We collect here several simple observations that help place the preceding formulas in context.

The first concerns intrinsic curvature. Although we do not attempt to compute the Gauss curvature explicitly for the general family, it is useful to recall that Gauss curvature is determined entirely by the first fundamental form. In particular, once the metric \eqref{metricformula} is known, the intrinsic curvature is determined independently of the ambient embedding.

\begin{proposition}
The Gauss curvature of the immersed surface is an intrinsic invariant, determined entirely by the metric.
\end{proposition}

\begin{proof}
This is precisely the content of Gauss' \emph{Theorema Egregium}. Since the metric determines the Levi-Civita connection and hence the Riemann curvature tensor of the surface, it determines the Gauss curvature.
\end{proof}

\begin{figure}
    \centering
    \includegraphics[width=0.95\linewidth]{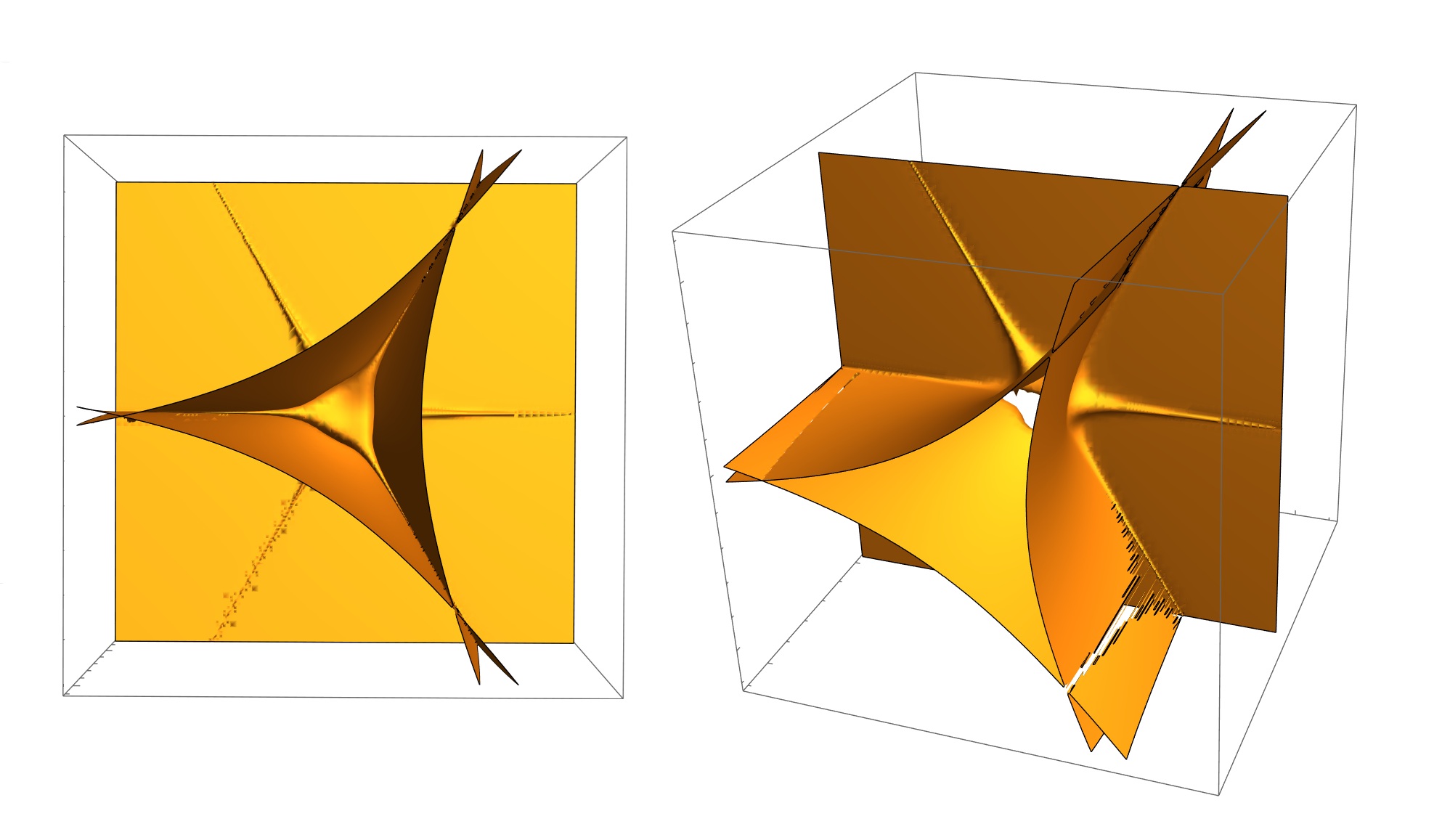}
    \caption{Two views from different angles of the implicit surface given by \eqref{eq:F} in the $X_1X_2X_3$-space.}
    \label{fig:projectionX123}
\end{figure}

Another question is when the construction degenerates to a plane. The answer is completely elementary.

\begin{proposition}\label{prop:plane}
If the holomorphic null data $\Phi$ is constant, then the associated immersion is an affine plane in $\mathbb{R}^5$.
\end{proposition}

\begin{proof}
If $\Phi$ is constant, then
\[
X(z)=2\,\Re(\Phi z)+\text{constant}.
\]
Writing $z=u+iv$, we obtain
\[
X(u,v)=2u\,\Re(\Phi)-2v\,\Im(\Phi)+\text{constant},
\]
which is an affine linear map of the plane into $\mathbb{R}^5$. Hence its image is an affine plane.
\end{proof}

In the seed construction, this happens precisely when $\Psi'''$ is constant.

\begin{corollary}
If $\Psi$ is a polynomial of degree at most three, then the associated immersion is planar.
\end{corollary}

\begin{proof}
If $\deg\Psi\leq 3$, then $\Psi'''$ is constant, so the null curve \eqref{phiseed1}--\eqref{phiseed345} is affine in $z$. The corresponding immersion is therefore planar by Proposition \ref{prop:plane}.
\end{proof}

\begin{remark}
Thus the first genuinely nonlinear examples in the present family arise at degree four. This is one reason the quartic seed is a natural starting point for visualization and experimentation.
\end{remark}

\section{On the Impossibility of a Naive Quaternionic Reformulation}

In view of possible extensions of the construction, it is important to clarify why quaternionic variables cannot be introduced in a naive way into the present Weierstrass-type representation.

At first sight, one might attempt to replace the complex parameter
\[
z = u + iv
\]
by a quaternionic variable
\[
q = u + \mathbf{i}v + \mathbf{j}s + \mathbf{k}t,
\]
and then try to reproduce the same formulas using quaternion-valued functions. However, this approach is fundamentally incompatible with the structure of the representation developed in this paper. The obstruction is not technical but structural, and it arises from three independent sources.

\subsection*{(1) Failure of commutativity}

The Weierstrass-type formula in $\mathbb{R}^5$ relies crucially on algebraic identities such as
\[
\phi_1^2 + \phi_2^2 + \phi_3^2 + \phi_4^2 + \phi_5^2 = 0,
\]
which are verified using the commutativity of complex multiplication. For instance, expressions of the form
\[
(1-\sigma)^2 - (1+\sigma)^2 = -4\sigma
\]
are used repeatedly in the nullity computation.

In the quaternionic setting, multiplication is noncommutative:
\[
ab \neq ba \quad \text{in general}.
\]
As a consequence, even the basic expansion
\[
(1-\sigma)^2 = 1 - \sigma - \sigma + \sigma^2
\]
depends on the order of multiplication, and cancellations that are automatic over $\C$ no longer occur. In particular, the nullity condition is no longer preserved by the same algebraic mechanism.

\subsection*{(2) Loss of complex linearity and holomorphic structure}

The construction of minimal surfaces here depends on the fact that
\[
X(z) = \Re \int \Phi(z)\,dz
\]
is defined using holomorphic data. This uses several essential features of complex analysis:

\begin{itemize}
\item the Cauchy--Riemann equations,
\item the existence of local primitives for holomorphic functions,
\item path-independence of integrals on simply connected domains,
\item and the identity $X_z = \Phi$, $X_{\bar z}=0$.
\end{itemize}

These properties imply harmonicity and conformality.

In contrast, quaternionic analysis does not provide a direct analogue of this structure. There is no single notion of quaternionic holomorphicity that simultaneously preserves:

\begin{itemize}
\item linearity over the base field,
\item compatibility with multiplication,
\item and a simple primitive theory analogous to $\int f(z)\,dz$.
\end{itemize}

As a result, the identity
\[
X_z = \Phi, \qquad X_{\bar z} = 0
\]
has no meaningful quaternionic counterpart, and the standard proof of minimality breaks down.

\subsection*{(3) Incompatibility with the null curve model}

The entire construction is based on holomorphic null curves in $\C^5$, that is, maps satisfying
\[
\sum_{k=1}^5 \phi_k^2 = 0.
\]
This condition is quadratic and \emph{complex-bilinear}. It is preserved under complex scalar multiplication and behaves well under projectivization, leading to the quadric $Q_3 \subset \CP^4$.

Quaternionic multiplication does not preserve such quadratic complex-bilinear relations. In particular, there is no natural quaternionic analogue of the complex null cone
\[
\{ w \in \C^5 : \sum w_k^2 = 0 \}
\]
with the same algebraic and geometric properties. Consequently, the link between null curves, conformality, and minimality cannot be transferred to the quaternionic setting by a simple substitution.

\subsection*{Conclusion}

The Weierstrass-type representation used in this paper is inherently complex-analytic. Its validity depends on commutativity, holomorphicity, and the algebraic structure of the complex null cone. Replacing the complex variable by a quaternionic one breaks each of these features simultaneously.

For this reason, quaternionic methods - while extremely valuable in other areas of differential geometry and integrable systems - cannot be introduced into the present construction in a direct or formal way. Any meaningful quaternionic reformulation would require a fundamentally different framework, rather than a modification of the formulas presented here.

\section{A Cartan-Theoretic Frame and a Local DPW Interpretation}

The explicit formulas developed so far are useful in their own right, but they also fit into a broader geometric framework. In this section we explain that framework at a theoretical level. Our aim is not to establish a full loop-group theorem, but rather to show how the null-curve construction can be reinterpreted in terms of moving frames, the generalized Gauss map, and a local DPW-type reconstruction.

Let
\[
X:D\to \mathbb{R}^5
\]
be a conformal minimal immersion with induced metric
\[
ds^2=e^{2u}|dz|^2.
\]
Choose a local orthonormal tangent frame $(e_1,e_2)$ and a local orthonormal normal frame $(n_1,n_2,n_3)$. These vectors determine an adapted frame
\[
\mathcal{F}=(e_1,e_2,n_1,n_2,n_3):D\to SO(5).
\]
Its differential is encoded in the Maurer-Cartan form
\[
\alpha = \mathcal{F}^{-1}\dd\mathcal{F}\in\Omega^1(D,\mathfrak{so}(5)).
\]
Relative to the splitting of $\mathbb{R}^5$ into tangent and normal parts, this matrix-valued $1$-form has the block form
\begin{equation}\label{MCblock}
\alpha
=
\begin{pmatrix}
\omega & -\beta^{\,t}\\
\beta & \nu
\end{pmatrix}.
\end{equation}
Here $\omega$ is the Levi-Civita connection form of the tangent bundle, $\nu$ is the normal connection form, and $\beta$ encodes the second fundamental form.

The structure equations of the immersion are equivalent to the single Maurer-Cartan equation
\begin{equation}\label{MCeq}
\dd\alpha+\alpha\wedge\alpha=0.
\end{equation}
When the block form \eqref{MCblock} is substituted into \eqref{MCeq}, one recovers the Gauss, Codazzi, and Ricci equations.

To bring this into closer contact with the holomorphic formalism, we pass to complex notation. Define
\[
\varepsilon=e_1-i e_2.
\]
Since $X$ is conformal and the metric is $e^{2u}|dz|^2$, one has
\[
X_z=\frac{e^u}{2}\varepsilon.
\]
The vector $\varepsilon$ is complex isotropic in the sense that
\[
\ip{\varepsilon}{\varepsilon}=0,
\qquad
\ip{\varepsilon}{\overline{\varepsilon}}=2.
\]
Differentiating once more gives a decomposition
\begin{equation}\label{Hopfdecomp}
X_{zz}=2u_z X_z+\Omega,
\end{equation}
where $\Omega$ is a section of the complexified normal bundle $N\Sigma\otimes\C$. This normal-valued quadratic differential is the higher-codimensional analogue of the classical Hopf differential. Minimality is expressed by the condition
\[
X_{z\bar z}=0.
\]

Now recall that in our Weierstrass representation we have
\[
X_z=\Phi.
\]
Thus the null curve itself is a holomorphic lift of the complex tangent direction. The complex line generated by $\Phi(z)$ determines the generalized Gauss map \eqref{Gdef},
\[
\mathcal{G}:D\to Q_3\subset\CP^4.
\]
Because $Q_3$ is a complex quadric, it may be viewed as the complex-analytic model of the set of isotropic lines in $\C^5$.

There is also a real-geometric interpretation. The oriented tangent plane of the immersion defines a map into the Grassmannian of oriented $2$-planes in $\mathbb{R}^5$, which is the compact symmetric space
\[
\widetilde{G}_2(\mathbb{R}^5) \cong \frac{SO(5)}{SO(2)\times SO(3)}.
\]
Accordingly, the Lie algebra $\mathfrak{so}(5)$ admits a Cartan decomposition
\[
\mathfrak{so}(5)=\mathfrak{k}\oplus\mathfrak{p},
\]
where $\mathfrak{k}\cong\mathfrak{so}(2)\oplus\mathfrak{so}(3)$ is the Lie algebra of the isotropy subgroup. Splitting the Maurer-Cartan form into its $\mathfrak{k}$- and $\mathfrak{p}$-components gives
\[
\alpha=\alpha_{\mathfrak{k}}+\alpha_{\mathfrak{p}}.
\]
After complexification and decomposition into $(1,0)$- and $(0,1)$-parts, one obtains
\[
\alpha_{\mathfrak{p}}
=
\alpha_{\mathfrak{p}}'+\alpha_{\mathfrak{p}}''.
\]

As in the standard theory of harmonic maps into symmetric spaces, one then introduces the associated family of connections
\begin{equation}\label{extendedconnection}
\alpha_\lambda
=
\alpha_{\mathfrak{k}}
+
\lambda^{-1}\alpha_{\mathfrak{p}}'
+
\lambda\,\alpha_{\mathfrak{p}}'',
\qquad
\lambda\in\C^\times.
\end{equation}
The key point is that flatness of this family is equivalent to harmonicity of the Gauss map.

\begin{proposition}\label{prop:flatfamily}
If the Gauss map of the immersion is harmonic as a map into
\[
SO(5)/(SO(2)\times SO(3)),
\]
then for every $\lambda\in\C^\times$ the connection $\dd+\alpha_\lambda$ is flat.
\end{proposition}

\begin{proof}
This is the standard zero-curvature formulation for harmonic maps into symmetric spaces. Substituting \eqref{extendedconnection} into the flatness condition
\[
\dd\alpha_\lambda+\alpha_\lambda\wedge\alpha_\lambda=0
\]
and comparing coefficients of the powers of $\lambda$ gives exactly the harmonic map equations for the symmetric-space-valued Gauss map, together with the Maurer-Cartan equation for $\alpha$. Conversely, the harmonic map equations imply flatness of the family.
\end{proof}

In our setting, the generalized Gauss map is holomorphic into the complex quadric $Q_3$ by Proposition \ref{prop:gaussmap}. Since holomorphic maps into K\"ahler manifolds are harmonic, the associated family \eqref{extendedconnection} is flat. This is the starting point for a local DPW-type description.

The idea of the DPW method is to begin not with the immersion itself, but with a meromorphic potential. In the present context one expects a local potential of the form
\begin{equation}\label{potential}
\xi = \lambda^{-1}\eta(z)\,\dd z,
\end{equation}
where $\eta(z)$ takes values in the complexified $\mathfrak{p}$-part of the Cartan decomposition. One then solves
\begin{equation}\label{ODE}
\dd C = C\,\xi.
\end{equation}
Locally, and away from singularities, one seeks a factorization
\begin{equation}\label{Iwasawa}
C(z,\lambda)=\mathcal{F}(z,\lambda)\,V_+(z,\lambda),
\end{equation}
where $\mathcal{F}(z,\lambda)$ takes values in the appropriate real loop group and $V_+(z,\lambda)$ belongs to the positive loop subgroup. The Maurer-Cartan form of $\mathcal{F}$ then reproduces the associated family \eqref{extendedconnection}.

What remains is to reconstruct the immersion from the Gauss data. This is where the null-curve formulation becomes especially transparent. The null curve $\Phi$ may be written as
\begin{equation}\label{Phimu}
\Phi=\mu\,\nu,
\end{equation}
where $\nu$ is a local lift of the projective Gauss map and $\mu$ is a scalar holomorphic differential. In our explicit family one may take
\[
\mu=\frac{\Psi'''(z)}{2}\,\dd z
\]
and
\[
\nu=
\left(
1-\Lambda z^2,\,
i(1+\Lambda z^2),\,
2z,\,
2\lambda_1 z,\,
2\lambda_2 z
\right).
\]
The immersion is then recovered by
\[
X(z)=2\,\Re\int_{z_0}^{z}\Phi
=
2\,\Re\int_{z_0}^{z}\mu\,\nu.
\]
Thus the geometric data naturally splits into a projective part, represented by the generalized Gauss map, and a scalar part, represented by the height differential $\mu$.

This discussion may be summarized as follows.

\begin{theorem}[Local DPW-type reconstruction principle]
Locally on a simply connected domain, a conformal minimal immersion in $\mathbb{R}^5$ determines a holomorphic generalized Gauss map into the quadric $Q_3$, hence a harmonic map into the symmetric space
\[
SO(5)/(SO(2)\times SO(3)),
\]
together with a flat family of connections of the form \eqref{extendedconnection}. Conversely, a local meromorphic potential producing such a flat family, together with a holomorphic scalar differential, determines the immersion through the reconstruction formula
\[
X(z)=2\,\Re\int_{z_0}^{z}\mu\,\nu.
\]
\end{theorem}

\begin{proof}
The forward direction follows from the preceding paragraphs: the null curve determines the holomorphic map into $Q_3$, holomorphicity implies harmonicity, harmonicity yields the flat associated family, and \eqref{Phimu} gives the reconstruction. The converse direction is the local content of the loop-group method: a suitable potential yields an extended frame, the frame determines the harmonic Gauss map, and the additional scalar differential reconstructs the null curve and hence the immersion. Since we are working locally, no monodromy issues arise.
\end{proof}

\begin{remark}
We stress that the purpose of this section is conceptual rather than technical. A complete global DPW theory for higher-codimensional minimal immersions requires additional analytic and loop-group input. What we have shown here is that the explicit formulas derived earlier are not isolated algebraic accidents; they fit naturally into the standard Cartan-harmonic-map framework.
\end{remark}

\subsection*{What the DPW viewpoint adds in the present setting}

At this stage it is useful to explain more explicitly what is, and what is not, meant by the DPW viewpoint in the present paper. The classical Dorfmeister-Pedit-Wu method was developed as a loop-group construction for harmonic maps into symmetric spaces and, through them, for several classes of geometrically significant surfaces. In its full form, the method begins with a meromorphic potential, solves a matrix differential equation, performs an Iwasawa or Birkhoff factorization, and then reconstructs the geometric object from the resulting extended frame.

Our discussion above does not attempt to reproduce that full theory in complete analytic generality. Instead, it identifies the precise point at which the explicit $\mathbb{R}^5$ Weierstrass data fits into the DPW framework. The key observation is that the generalized Gauss map of our immersion is holomorphic into the quadric $Q_3$, and hence harmonic when viewed as a map into the corresponding symmetric-space model. This harmonicity is exactly what permits the introduction of the associated family of flat connections \eqref{extendedconnection}. Once that family is present, the logic of the DPW method becomes available at the local level.

What makes the present situation especially transparent is that the null curve already provides a concrete factorization of the geometric data. Indeed, the formula
\[
\Phi=\mu\,\nu
\]
separates the construction into two pieces. The vector $\nu$ determines the projective direction of the isotropic tangent line and hence the generalized Gauss map, while the scalar differential $\mu$ measures the size of the null curve and is responsible for the actual immersion after integration. In other words, the projective part of the theory belongs to the harmonic-map side, while the scalar differential plays the role of the height data needed to pass from a Gauss map to a surface.

This makes it possible to interpret the explicit formulas of Sections 2--5 as a particularly tractable class of DPW-type data. In the family generated from a single seed function $\Psi$, the generalized Gauss map is already available in closed form through \eqref{GaussExplicit}, and the scalar differential is simply
\[
\mu=\frac{\Psi'''(z)}{2}\,\dd z.
\]
Thus no abstract existence argument is needed to identify the potential geometric ingredients: they are already visible directly in the formulas. From this point of view, the seed function controls the conformal scaling, while the quadric-valued map records the isotropic direction field.

There is, however, an important distinction between this observation and a complete DPW theorem. A full loop-group treatment would require a precise choice of real form, a careful discussion of the relevant loop groups, local and global factorization issues, and an analysis of singularities and monodromy. None of these technical points is needed for the explicit goals of the present paper, and we have therefore deliberately avoided burdening the exposition with them. Our purpose is not to replace the concrete Weierstrass formulas by an abstract machine, but rather to show that those formulas sit naturally inside a broader integrable-systems picture.

For that reason, the role of the DPW viewpoint here is primarily explanatory. It clarifies why the generalized Gauss map, the frame equations, and the null-curve representation are not unrelated constructions, but different manifestations of the same underlying geometry. It also explains why the family introduced from the seed function $\Psi$ is more than an ad hoc collection of examples: it can be regarded as a concrete slice of a larger harmonic-map and loop-group framework. In this sense, the DPW discussion closes the present study by connecting the explicit formulas back to the structural theory from which they ultimately derive.

\section{Concluding Remarks}

We have presented an explicit and local framework for conformal minimal immersions in $\mathbb{R}^5$ based on holomorphic null curves in $\C^5$. The five-dimensional setting is rich enough to display genuinely higher-codimensional behavior, but still structured enough that one can carry out the relevant computations completely and explicitly.

The first part of the paper developed a Weierstrass-type formula and then specialized it to a family determined by a single holomorphic seed function and two real parameters. The resulting integral-free parametrization makes it possible to write the immersion, the metric, and the generalized Gauss map in closed form. This gives a concrete class of examples that can be handled directly, without leaving the level of elementary holomorphic computations.

The polynomial case illustrates the usefulness of this explicit approach. Even very simple seeds produce coordinates with several competing powers of the complex variable, and this leads to nontrivial geometry and interesting projected images. At the same time, the formulas remain transparent enough that a reader can verify each step and adapt the construction to further examples.

The final section showed that these explicit formulas admit a broader geometric interpretation. The moving-frame viewpoint, the generalized Gauss map into the quadric, and the associated family of flat connections place the entire construction within the standard language of Cartan theory and integrable systems. In this way the paper connects explicit computation with conceptual structure.

Several questions naturally remain open. One may choose more general functions $g_1,g_2,g_3$, examine global period conditions, study completeness and singularities, or pursue a fuller loop-group treatment. These directions lie beyond the local and formula-driven aims of the present work. Our hope is that the paper provides a clear and useful entry point into minimal surface geometry in higher codimension, and especially into the concrete study of explicit examples in $\mathbb{R}^5$.

\section*{Acknowledgements}

The authors are grateful to the classical literature on minimal surfaces and higher-codimensional surface theory, which continues to provide both inspiration and conceptual guidance for explicit


\textbf{Magdalena Toda}\\
Department of Mathematics and Statistics, Texas Tech University,\\
Lubbock, TX 79409, USA\\
\textit{Email:} \texttt{magda.toda@ttu.edu}\\[0.1em]

\textbf{Erhan Güler}\\
Department of Mathematics and Statistics, Texas Tech University,\\
Lubbock, TX 79409, USA\\
\textit{Email:} \texttt{eguler@ttu.edu}

\end{document}